\documentclass[12pt,reqno]{amsart}

\usepackage{multirow}
\usepackage{hhline}

\usepackage{mathtools}
\usepackage{times}
\usepackage[T1]{fontenc}
\usepackage{mathrsfs}
\usepackage{latexsym}
\usepackage[dvips]{graphics}
\usepackage[titletoc, title]{appendix}
\usepackage{amsmath,amsfonts,amsthm,amssymb,amscd}
\usepackage[dvipsnames]{xcolor}
\usepackage{hyperref}
\usepackage{amsmath}
\usepackage[utf8]{inputenc}

\usepackage{color}
\usepackage{breakurl}

\usepackage{comment}
\newcommand{\bburl}[1]{\textcolor{blue}{\url{#1}}}

\newtheorem{thm}{Theorem}[section]

\newtheorem{claim}[thm]{Claim}
\newtheorem{lem}[thm]{Lemma}
\newtheorem{prop}[thm]{Proposition}
\newtheorem{exa}[thm]{Example}

\usepackage[utf8]{inputenc}

\DeclareFixedFont{\ttb}{T1}{txtt}{bx}{n}{12} 
\DeclareFixedFont{\ttm}{T1}{txtt}{m}{n}{12}  

\usepackage{color}
\definecolor{deepblue}{rgb}{0,0,0.5}
\definecolor{deepred}{rgb}{0.6,0,0}
\definecolor{deepgreen}{rgb}{0,0.5,0}

\usepackage{listings}

\newcommand\pythonstyle{\lstset{
language=Python,
basicstyle=\ttm,
morekeywords={self},              
keywordstyle=\ttb\color{deepblue},
emph={MyClass,__init__},          
emphstyle=\ttb\color{deepred},    
stringstyle=\color{deepgreen},
frame=tb,                         
showstringspaces=false
}}

\lstnewenvironment{python}[1][]
{
\pythonstyle
\lstset{#1}
}
{}

\newcommand\pythoninline[1]{{\pythonstyle\lstinline!#1!}}

\numberwithin{equation}{section}

\DeclareFontFamily{U}{mathx}{}
\DeclareFontShape{U}{mathx}{m}{n}{<-> mathx10}{}
\DeclareSymbolFont{mathx}{U}{mathx}{m}{n}
\DeclareMathAccent{\widehat}{0}{mathx}{"70}
\DeclareMathAccent{\widecheck}{0}{mathx}{"71}

\begin{document}

\title{Weighted Schreier-type Sets and the Fibonacci Sequence}
\author{H\`ung Vi\d{\^e}t Chu}
\address{Department of Mathematics, Texas A\&M University, College
  Station, TX 77843, USA}
  \email{hungchu1@tamu.edu}
\author{Zachary Louis Vasseur}
\address{Department of Mathematics, Texas A\&M University, College
  Station, TX 77843, USA}
\email{zachary.l.v@tamu.edu}
\thanks{The work was partially supported by the College of Arts \& Sciences at Texas A\&M University.}



\maketitle

\begin{abstract}
For a finite set $A\subset\mathbb{N}$ and $k\in \mathbb{N}$, let $\omega_k(A) = \sum_{i\in A, i\neq k}1$. For each $n\in \mathbb{N}$, define
$$a_{k, n}\ =\ |\{E\subset \mathbb{N}\,:\, E = \emptyset\mbox{ or } \omega_k(E) < \min E\leqslant \max E\leqslant n\}|.$$
First, we prove that 
$$a_{k,k+\ell} \ =\ 2F_{k+\ell},\mbox{ for all }\ell\geqslant 0\mbox{ and }k\geqslant \ell+2,$$ where $F_n$ is the $n$th Fibonacci number. Second, we show that
$$|\{E\subset \mathbb{N}\,:\, \max  E = n+1, \min E > \omega_{2,3}(E), \mbox{ and }|E|\neq 2\}|\ =\ F_{n},$$
where $\omega_{2,3}(E) = \sum_{i\in E, i\neq 2, 3}1$.
\end{abstract}

\section{Introduction}
Recall that a finite set $E\subset \mathbb{N}$ is called a \textit{Schreier} set if $\min E\geqslant |E|$. Denote the collection of all Schreier sets by $\mathcal{S}$, which includes the empty set. Sets $E$ with $\min E > |E|$ ($\min E = |E|$, respectively) are called \textit{nonmaximal Schreier} sets (\textit{maximal Schreier} sets, respectively). The empty set is vacuously both maximal and nonmaximal. We let $S^{NMAX}$ consist of nonmaximal Schreier sets and let $S^{MAX}$ consist of maximal Schreier sets. Bird \cite{Bi} showed that for each positive integer $n$, 
$$|\{E\in \mathcal{S}\,:\, \max E = n\}| \ =\ F_n,$$ 
where $(F_n)_{n=0}^\infty$ is the Fibonacci sequence defined as: $F_0 = 0, F_1 = 1$ and $F_n = F_{n-1} + F_{n-2}$ for $n\geqslant 2$. Beanland et al. \cite{BCF} generalized the Schreier condition to $q\min E\geqslant p|E|$ and established an inclusion-exclusion type recurrence for the sequence
\begin{equation}\label{e46}m_{p, q, n} \ :=\ |\{E\subset \mathbb{N}\,:\, q\min E \geqslant p|E| \mbox{ and }\max E = n\}|.\end{equation}
In particular, \cite[Theorem 1]{BCF} gives
$$m_{p,q,n}\ =\ \sum_{k=1}^q(-1)^{k+1}\binom{q}{k}m_{p, q, n-k}+ m_{p,q,n-(p+q)},\mbox{ for }n\geqslant p+q.$$
Recently, Beanland et al.\ \cite{BGHH} studied unions of Schreier sets and proved a linear recurrence for their counts using recursively defined characteristic polynomials. Relations between Schreier sets and partitions, compositions, and Tur\'{a}n graphs have also been discovered \cite{BC, C3, CIMSZ}.

Previous work gave a uniform weight to  each number when measuring the size of a set. For example, rewriting $q\min E\geqslant p|E|$ in \eqref{e46} as $\min E\geqslant (p/q)|E|$, one can think of $(p/q)|E|$  as the size of $E$, where each element is given the same weight $p/q$. For our first results, we assign the weight $0$ to exactly one number and the weight $1$ to the other numbers. Specifically, let $\mathcal{N}$ be the collection of finite subsets of $\mathbb{N}$. For $E\in \mathcal{N}$ and $k\in \mathbb{N}$, define  $\omega_k (E) := \sum_{n\in E, n\neq k}1$ and 
$$\mathcal{S}^{(k)} \ :=\ \{E\in \mathcal{N}\,:\, E = \emptyset \mbox{ or }\min E > \omega_k(E)\}.$$
For example, $\mathcal{S}^{(1)}  = \mathcal{S}^{NMAX}\cup \{1\}$, $\mathcal{S}^{(2)}  = \mathcal{S}^{NMAX}\cup \{\{2,n\}: 2 < n\}$, and $\mathcal{S}^{(3)}  = \mathcal{S}^{NMAX}\cup \{\{2,3\}\}\cup \{\{3, n_1, n_2\}: 3 < n_1 < n_2\}$.
Generally, we have the following equality, whose proof is in the appendix
\begin{equation}\label{e1}\mathcal{S}^{(k)} \ =\ \mathcal{S}^{NMAX}\cup \{E\in \mathcal{S}^{MAX}\,:\, k\in E\}.\end{equation}

For $n, k\in \mathbb{N}$, let 
$$\mathcal{A}_{k,n} \ :=\ \{E\in \mathcal{S}^{(k)}\,:\, E = \emptyset \mbox{ or }\max E \leqslant n\}\mbox{ and }a_{k,n} = |\mathcal{A}_{k,n}|.$$ 
We can write
$$\mathcal{A}_{k, n}\ =\ \{E\in \mathcal{N}\,:\, E = \emptyset\mbox{ or } \omega_k(E) < \min E\leqslant \max E\leqslant n\}.$$
The following relates $a_{n,n}$ to the Fibonacci numbers.
\begin{thm}\label{m2}
For $n\in \mathbb{N}$, it holds that
\begin{equation}
\label{eFibn} a_{n,n} \ =\ 2F_n.
\end{equation}
\end{thm}

Theorem \ref{m2} is used to establish a general formula for $a_{k,n}$ for each pair $(k,n)\in \mathbb{N}^2$. First, let us record values of $a_{k,n}$ for small $k$ and $n$.

\begin{tabular}{c|cccccccccccccccc}
$k\backslash n$ & $1$ & $2$ & $3$ & $4$ & $5$ & $6$ & $7$ & $8$ & $9$ & $10$ & $11$ & $12$& $13$ & $14$ & $15$ & $16$\\
\hline
$1$ & $2$ & $3$ & $4$ & $6$        & $9$ &    $14$ &   $22$ & $35$ & $56$ & $90$ & $145$ & $234$ & $378$ & $611$ & $988$ & $1598$\\
$2$ & $1$ & $2$ & $4$ & $7$    & $11$ &   $17$&   $26$ & $40$ & $62$ & $97$ & $153$ & $243$ & $388$ & $622$ & $1000$ & $1611$\\
$3$ & $1$ & $2$ & $4$ & $6$   & $10$ &   $17$&   $28$ & $45$ & $71$ & $111$ & $173$ & $270$ & $423$ & $666$ & $1054$ & $1676$\\
$4$ & $1$ & $2$ & $3$ & $6$    & $10$ &   $16$&   $26$ & $43$ & $71$ & $116$ & $187$ & $298$ & $471$ & $741$ & $1164$ & $1830$\\
$5$ & $1$ & $2$ & $3$ & $5$  & $10$ &   $16$&  $26$ & $42$ & $68$ & $111$ & $182$ & $298$ & $485$ & $783$ & $1254$ & $1995$\\
$6$ & $1$ & $2$ & $3$ & $5$  & $8$ &  $16$&  $26$ & $42$ & $68$ & $110$ & $178$ & $289$  & $471$ & $769$ & $1254$ & $2037$\\
$7$ & $1$ & $2$ & $3$ & $5$  & $8$ &  $13$&  $26$ & $42$ & $68$ & $110$ & $178$ & $288$  & $466$ & $755$ & $1226$ & $1995$
\end{tabular}

\begin{center}
Table 1. Initial numbers $a_{k, n}$ for different $k$ and $n$.
\end{center}

Observe that when $k > n$, $a_{k, n}$'s are Fibonacci numbers. This is expected because when $k > n$, $\omega_k(E) = |E|$ for all sets $E\subset \{1, \ldots, n\}$, so the weight $\omega_k$ gives the cardinality of sets as in the result by Bird \cite{Bi}. When $k = n$, we witness the Fibonacci numbers multiplied by $2$, which is Theorem \ref{m2}. When $n > k$, we shall see that $a_{k,n}$'s are found by iterated partial sums. The next theorem summarizes these observations. 

\begin{thm}\label{m4} It holds that
$$a_{k, k+ \ell}\ =\ \begin{cases} F_{\ell+2}+1, &\mbox{ if }k = 1, \ell\geqslant 0,\\
2\sum_{i=0}^{k-2}\binom{\ell}{i}F_{k-i} + 2\binom{\ell}{k-1} + \sum_{j=1}^{\ell}\binom{j}{\ell-j+k}, &\mbox{ if }k\geqslant 2, \ell\geqslant 0,\\ F_{k+\ell+1}, &\mbox{ if }k\geqslant 2, -k < \ell < 0.\end{cases}$$
\end{thm}

\begin{exa}\normalfont
Let $k = 4$ and $\ell = 6$. Table 1 gives us $a_{4, 10} = 116$, while Theorem \ref{m4} gives
\begin{align*}&2\sum_{i=0}^{2}\binom{6}{i}F_{4-i} + 2\binom{6}{3} + \sum_{j=1}^{6}\binom{j}{10-j}\\
&\ =\ 2\left(\binom{6}{0}F_4+\binom{6}{1}F_3+\binom{6}{2}F_2\right) + 2\binom{6}{3} + \binom{5}{5} + \binom{6}{4}\ =\ 116, 
\end{align*}
as expected. 
\end{exa}

As a corollary of Theorem \ref{m4}, we obtain a more general identity than the one in Theorem \ref{m2}.

\begin{thm}\label{c1}
For $\ell\geqslant 0$ and $k\geqslant \ell+2$, we have
$$a_{k, k+\ell}\ =\ 2F_{k+\ell}.$$
\end{thm}

For our last result, we put the zero weight on both $2$ and $3$ and again observe the appearance of the Fibonacci numbers. 
For $E\in\mathcal{N}$, define $\omega_{2,3}(E) = \sum_{n\in E\backslash\{2, 3\}} 1$. 
\begin{thm}\label{m3}
For $n\in \mathbb{N}$, let 
$$\mathcal{K}_n\ :=\ \{E\in \mathcal{N}\,:\, \max  E = n, \min E > \omega_{2,3}(E), \mbox{ and }|E|\neq 2\}.$$
Then $|\mathcal{K}_{n+1}| = F_{n}$ for all $n\in \mathbb{N}$.
\end{thm}

Our paper is structured as follows: Section \ref{nn2F_n} proves Theorem \ref{m2}, while Section \ref{table} uses Theorem \ref{m2} and properties of iterated partial sums to prove Theorem \ref{m4}; finally, Section \ref{KnFn} proves Theorem \ref{m3}.

\section{The sequence $(a_{n,n})_{n\geqslant 1}$} \label{nn2F_n}

We shall present two proofs of Theorem \ref{m2}, the first of which employs bijective maps, while the second utilizes the two well-known formulas:
\begin{align}
    \label{a1}\sum_{i=m}^{n}\binom{i}{m}&\ =\ \binom{n+1}{m+1}, \mbox{ for }n, m\in \mathbb{N}\mbox{, and }\\
    \label{a2}\sum_{k=0}^{\lfloor n/2\rfloor}\binom{n-k}{k}&\ =\ F_{n+1}, \mbox{ for }n\geqslant 0.
\end{align}

\begin{proof}[Bijective proof of Theorem \ref{m2}]
Since $a_{1,1} = a_{2,2} = 2$, it suffices to show that $a_{n+1,n+1} = a_{n,n} + a_{n-1,n-1}$ for $n\geqslant 2$. 
Consider the following map $\Psi_1: \mathcal{A}_{n-1,n-1}\rightarrow \mathcal{A}_{n+1,n+1}$ defined as
$$\Psi_1(F) \ =\ (F+1)\cup \{n+1\}.$$
Then $\Psi_1$ is well-defined because $\Psi_1(\emptyset) = \{n+1\}\in \mathcal{A}_{n+1, n+1}$, and if $F\neq \emptyset$, 
\begin{equation}\label{e20}\min \Psi_1(F)\ =\ \min F + 1 \ \geqslant\ |F| + 1\ =\ |\Psi_1(F)|\ >\ |\Psi_1(F)| - 1 \ =\ \omega_{n+1}(\Psi_1(F)).\end{equation}
Furthermore, $\Psi_1(F)$ is clearly one-to-one. 

Define the following map $\Psi_2: \mathcal{A}_{n,n}\rightarrow \mathcal{A}_{n+1, n+1}\backslash \Psi_1(\mathcal{A}_{n-1,n-1})$ as 
$$\Psi_2(F)\ =\ \begin{cases}F,&\mbox{ if } F\in \mathcal{S}^{NMAX};\\(F\backslash \{n\})\cup \{n+1\},&\mbox{ if }F\in \mathcal{S}^{MAX}\backslash \{\emptyset\}.\end{cases}$$
We show that $\Psi_2$ is well-defined. Obviously, $\Psi_2(\mathcal{A}_{n,n})\subset \mathcal{A}_{n+1, n+1}$. Let us show that $\Psi_2(\mathcal{A}_{n,n})\cap \Psi_1(\mathcal{A}_{n-1, n-1}) = \emptyset$. Pick $F\in \mathcal{A}_{n,n}$.

Case 1: If $F\in \mathcal{S}^{NMAX}$, then $\Psi_2(F) = F$ and $n+1\notin \Psi_2(F)$, so $\Psi_2(F)\notin \Psi_1(\mathcal{A}_{n-1,n-1})$. 

Case 2: If $F\in \mathcal{S}^{MAX}\backslash \{\emptyset\}$, then \eqref{e1} implies that  $n\in F$. That $F$ contains $n\geqslant 2$ and $F\in \mathcal{S}^{MAX}$ imply that $|F|\geqslant 2$. 
Therefore, 
\begin{equation}\label{e30}|\Psi_2(F)| \ =\ |F| \ =\ \min F\ =\ \min \Psi_2(F).\end{equation}
Suppose, for a contradiction, that there exists $G\in \mathcal{A}_{n-1,n-1}$ such that $\Psi_1(G) = \Psi_2(F)$. By \eqref{e30}, $\min \Psi_1(G) = |\Psi_1(G)|$.  It follows from \eqref{e20} that $\min G = |G|$. By \eqref{e1}, we must have $n-1\in G$. Hence, $n\in \Psi_1(G)$, contradicting that $n\notin \Psi_2(F)$. We have shown that $\Psi_2(\mathcal{A}_{n,n})\cap \Psi_1(\mathcal{A}_{n-1, n-1}) = \emptyset$. 

Furthermore, $\Psi_2$ is one-to-one. Indeed, suppose that $\Psi_2(F_1) = \Psi_2(F_2)$. If $n+1\notin \Psi_2(F_1) = \Psi_2(F_2)$, then by the definition of $\Psi_2$, we have
$$F_1 \ =\ \Psi_2(F_1)\ =\ \Psi_2(F_2) \ =\ F_2.$$
If $n+1\in \Psi_2(F_1) = \Psi_2(F_2)$, then $F_1, F_2\in \mathcal{S}^{MAX}\backslash \{\emptyset\}$, so \eqref{e1} implies that $n\in F_1\cap F_2$. Therefore, $\Psi_2(F_1) = \Psi_2(F_2)$ guarantees that $F_1 = F_2$.

Finally, we show that $\Psi_2$ is surjective. Take a nonempty $F\in \mathcal{A}_{n+1, n+1}\backslash \Psi_1(\mathcal{A}_{n-1,n-1})$.

Case 1: $F\in \mathcal{S}^{NMAX}$. Then $n+1\notin F$ because otherwise, $\Psi_1(F\backslash \{n+1\}-1) = F$. 
By \eqref{e1}, $F\in \mathcal{S}^{NMAX}$, so $\Psi_2(F) = F$. 

Case 2: $F\in \mathcal{S}^{MAX}$. Then $n+1\in F$. Let $E = (F\backslash \{n+1\})\cup \{n\}$. It follows that $|F|\geqslant |E|$. Note that $F\notin \Psi_1(\mathcal{A}_{n-1,n-1})$ implies that $F\neq \{n+1\}$. Hence, $|F|\geqslant 2$ and so, 
$$\min E \ =\ \min F\ >\ |F|-1\ \geqslant\ |E|-1 \ =\ \omega_n(E).$$
Therefore, $E\in \mathcal{A}_{n,n}$. To show that $\Psi_2(E) = F$, it remains to verify that $n\notin F$. Suppose otherwise, i.e., $n\in F$. Let $H = F\backslash \{n+1\}-1$. We have
$$\min H \ =\ \min F - 1 \ >\ \omega_{n+1}(F) - 1 \ =\ |F| - 2 \ =\ |H| - 1\ =\ \omega_{n-1}(H).$$
Hence, $F =  \Psi_1(H)\in \Psi(\mathcal{A}_{n-1,n-1})$, a contradiction. This completes our proof. 
\end{proof}

\begin{proof}[Alternative proof of Theorem \ref{m2}]
For $n\in \mathbb{N}$, we build a set $A\subset \{1, \ldots, n\}$ with $\min A > \omega_n(A)$. 

If $n\notin A$, then $\omega_n(A) = |A|$. Set $k = \min A\in \{1, \ldots, n-1\}$. For each such $k$, the set $A\backslash \{k\}\subset \{k+1, k+2, \ldots, n-1\}$. Since $k > |A|$, $|A\backslash \{k\}|\leqslant k-2$. Hence, the number of sets $A\in\mathcal{A}_{n, n}$ with $n\notin A$ is 
$$c = \sum_{k=1}^{n-1}\sum_{j=0}^{k-2}\binom{n-k-1}{j}.$$

If $n\in A$, then $\omega_n(A) = |A|-1$. Set $k = \min A\in \{1, \ldots, n\}$. For each $k\neq n$, the set $A\backslash \{k, n\}$ is a subset of $\{k+1, k+2, \ldots, n-1\}$ and has size 
$$|A\backslash \{k, n\}|\ = \ |A|-2\ =\ (|A|-1)-1\ <\ k-1.$$ Hence, the number of sets $A\in \mathcal{A}_{n, n}$ with $n\in A$ and $A\neq \{n\}$ is also $c$. 

Including also the empty set and $\{n\}$, we obtain 
$$a_{n,n} \ =\ 2 + 2\sum_{k=1}^{n-1}\sum_{j=0}^{k-2}\binom{n-k-1}{j}.$$
It suffices to show that 
$$\sum_{k=1}^{n-1}\sum_{j=0}^{k-2}\binom{n-k-1}{j}\ =\ F_n-1.$$
Exchanging order of the double sum then use \eqref{a1} and \eqref{a2}, we have
\begin{align*}
\sum_{k=1}^{n-1}\sum_{j=0}^{k-2}\binom{n-k-1}{j}&\ =\ \sum_{j=0}^{n-3}\sum_{k=j+2}^{n-1}\binom{n-k-1}{j}\ =\ \sum_{j=0}^{\lfloor (n-3)/2\rfloor}\sum_{k=j}^{n-3-j}\binom{k}{j}\\
&\ =\ \sum_{j=0}^{\lfloor (n-3)/2\rfloor}\binom{n-2-j}{j+1}\ =\ \sum_{j=1}^{\lfloor (n-1)/2\rfloor}\binom{n-1-j}{j}\\
&\ =\ F_n - 1, 
\end{align*}
as desired. 
\end{proof}

\section{The table $(a_{k,n})_{k,n\geqslant 1}$}\label{table}

The goal of this section is to prove Theorem \ref{m4}, giving formulas to compute $a_{k,n}$ for any $(k,n)\in \mathbb{N}^2$. We start with an easy observation. 

\begin{prop}\label{kp1}
For $k > n\in \mathbb{N}$, we have
$$a_{k,n} \ =\ F_{n+1}, \mbox{ for }k > n.$$
\end{prop}
\begin{proof}
For $k > n$, we have 
\begin{align*}a_{k, n} \ =\ |\mathcal{A}_{k, n}| &\ =\ |\{E\in \mathcal{S}^{(k)}\,:\, E = \emptyset \mbox{ or }\max E \leqslant n\}|\\
&\ =\ |\{E\subset \{1, 2,\ldots, n\}\,:\, E = \emptyset\mbox{ or } \min E > \omega_k(E)\}|\\
&\ =\  |\{E\subset \{1, 2,\ldots, n\}\,:\, E = \emptyset\mbox{ or } \min E > |E|\}|\\
&\ =\ F_{n+1},
\end{align*}
where the last equality follows from \cite[Theorem 1]{C1}.
\end{proof}

\begin{prop}
For $n > \max\{k, 2\}$, we have
\begin{equation}\label{e10}a_{k, n} \ =\ a_{k, n-1}+ a_{k-1, n-2}.\end{equation}
\end{prop}
\begin{proof}
It is obvious from the definition of $\mathcal{A}_{k, n}$ that $\mathcal{A}_{k, n-1}\subset \mathcal{A}_{k,n}$. 
It suffices to show that there is a bijective map between $\mathcal{A}_{k,n}\backslash \mathcal{A}_{k,n-1}$ and $\mathcal{A}_{k-1, n-2}$. To do so, we define $\Psi: \mathcal{A}_{k-1, n-2}\rightarrow \mathcal{A}_{k,n}\backslash \mathcal{A}_{k,n-1}$ as
$$\Psi(F) \ :=\ (F+1)\cup \{n\}.$$

We check that $\Psi$ is well-defined. Let $F \in \mathcal{A}_{k-1,n-2}$ and $E=\Psi (F)$. If $F=\emptyset$, then $E=\{n\}\in\mathcal{A}_{k,n}\backslash\mathcal{A}_{k,n-1}$. If $F\neq\emptyset$, we have 
$\max E = n$ and $$\min E \ =\ \min F + 1 \ >\ \omega_{k-1}(F) + 1 \ =\ \omega_k(F+1)+1 \ =\ \omega_k((F+1)\cup\{n\}) \ =\  \omega_k(E).$$ 
Thus, $\Psi$ is well-defined. 

Clearly, $\Psi$ is injective. Let us verify that $\Psi$ is surjective. Pick $E\in \mathcal{A}_{k,n}\backslash \mathcal{A}_{k,n-1}$. Then $n\in E$. Let $F=E\backslash\{n\}-1$. We claim that $F\in\mathcal{A}_{k-1,n-2}$. Indeed, if $E=\{n\}$, then $F = \emptyset$. If $\{n\}\subsetneq E$, then 
$$\max F \ =\ \max E\backslash\{n\} - 1 \ \leqslant\ n-1-1\ =\ n-2.$$
Furthermore, 
$$\min F \ =\ \min E-1\ >\ \omega_k(E)-1\ =\ \omega_k((F+1)\cup\{n\})-1\ =\ \omega_k(F+1)\ =\ \omega_{k-1}(F).$$ This completes the proof.
\end{proof}

Next, we collect useful properties of iterated partial sums of a sequence. Given a sequence $(a_n)_{n=0}^\infty$ and a number $b$, the partial sum operator $P_b$ produces the sequence $(a_n')_{n=0}^\infty$ defined as follows:
$$a_{0}' \ =\ b, \quad a'_{1}\ =\ b+a_0, \quad a'_{2} \ =\ b+a_0 + a_1, \quad \ldots.$$
In general, 
$$P_b((a_n)_{n=0}^\infty) \ :=\ \left(a'_n := b+\sum_{i=0}^{n-1} a_i\right)_{n\geqslant 0}.$$
Fix a sequence $\mathbf s = (b_n)_{n=0}^\infty$ and let
$(t^{\mathbf s}_{k,n})_{n=0}^\infty$, for $k\geqslant 0$, denote the sequence
$$P_{b_{k-1}}(\cdots P_{b_1}(P_{b_0}((a_n)_{n=0}^\infty))).$$
Here $(t^\mathbf{s}_{0, n})_{n=0}^\infty = (a_n)_{n=0}^\infty$. 
In the special case that $\mathbf s$ consists of only $0$, the sequence 
$(t^{\mathbf s}_{k,n})_{n\geqslant 0}$ is what known as \textit{the $k$-partial sum of $(a_n)_{n\geqslant 0}$}, which shall be denoted by $P^{(k)}((a_n)_{n\geqslant 0})$, and its $m$th term is denoted by $P^{(k)}((a_n)_{n\geqslant 0})(m)$. As a result, $$P^{(k)}((a_n)_{n\geqslant 0})(m) \ =\ t^{\mathbf s}_{k,m}, \mbox{ for all }m\geqslant 0.$$

Let $\mathbf s_1$ consist of only $0$ and $\mathbf s_2 = (b_n)_{n\geqslant 0}$. Our following lemma compares, for each $k\in \mathbb{N}$, the two sequences $(t^{\mathbf s_1}_{k,n})_{n\geqslant 0}$ and $(t^{\mathbf s_2}_{k,n})_{n\geqslant 0}$. The proof of the lemma is standard, involving double induction. 

\begin{lem}\label{kl1}
Let $k\geqslant 0$ and let $\mathbf s_1$ consist of only $0$ and $\mathbf s_2 = (b_n)_{n\geqslant 0}$. Then 
\begin{equation}\label{ep1}
t^{\mathbf s_2}_{k,n} - t^{\mathbf s_1}_{k,n}\ =\ \sum_{i=0}^{k-1}\binom{n}{i}b_{k-1-i}, \mbox{ for all }n\geqslant 0.
\end{equation}
\end{lem}

\begin{proof}
We prove by induction on $k$. For the base case, \eqref{ep1} clearly holds for $k \in \{0,1\}$. Suppose that \eqref{ep1} holds for all $k\leqslant j$ for some $j\geqslant 1$. We shall show that \eqref{ep1} holds for $k = j+1$. To do so, we proceed by induction on $n\geqslant 0$, given that $k = j+1$ and that \eqref{ep1} holds for all $k\leqslant j$. For the base case, by definitions, we have
$$t^{\mathbf s_2}_{j+1,0} - t^{\mathbf s_1}_{j+1,0} \ =\  b_j - 0 \ =\ b_j\ =\ \sum_{i=0}^j\binom{0}{i}b_{j-i};$$
hence, \eqref{ep1} is true for $k = j+1$ and $n = 0$. Assume that \eqref{ep1} is true for $k = j+1$ and all $n\leqslant \ell$ for some $\ell\geqslant 0$. We show that \eqref{ep1} is true for $k = j+1$ and $n = \ell+1$. Indeed, by the two inductive hypotheses, we have
\begin{align*}
t^{\mathbf s_2}_{j+1,\ell+1} -  t^{\mathbf s_1}_{j+1,\ell+1}&\ =\ \left(t^{\mathbf s_2}_{j+1,\ell} + t^{\mathbf s_2}_{j,\ell}\right) - \left(t^{\mathbf s_1}_{j+1,\ell} + t^{\mathbf s_1}_{j,\ell}\right)\\
&\ =\ \left(t^{\mathbf s_2}_{j+1,\ell}-t^{\mathbf s_1}_{j+1,\ell}\right) + \left(t^{\mathbf s_2}_{j,\ell}-t^{\mathbf s_1}_{j,\ell}\right)\\
&\ =\ \sum_{i=0}^j \binom{\ell}{i}b_{j-i} + \sum_{i=0}^{j-1}\binom{\ell}{i}b_{j-1-i}\\
&\ =\ \sum_{i=0}^j \binom{\ell}{i}b_{j-i} + \sum_{i=1}^j \binom{\ell}{i-1}b_{j-i}\\
&\ =\ b_j + \sum_{i=1}^j \left(\binom{\ell}{i-1}+\binom{\ell}{i}\right)b_{j-i}\ =\ \sum_{i=0}^{j}\binom{\ell+1}{i}b_{j-i}.
\end{align*}
This completes our proof. 
\end{proof}

\begin{lem}\label{kl2}
Fix $k\geqslant 0$ and a sequence $(a_n)_{n\geqslant 0}$. Then 
\begin{equation}\label{e31}P^{(k)}((a_n+1)_{n\geqslant 0})(m)-P^{(k)}((a_n)_{n\geqslant 0})(m)\ =\ \binom{m}{k}, \mbox{ for all }m\geqslant 0.\end{equation}
\end{lem}

\begin{proof}
We prove by induction on $k$. Base case: \eqref{e31} is clearly true when $k = 0$. Inductive hypothesis: assume that \eqref{e31} holds for $k\leqslant \ell$ for some $\ell\geqslant 0$. We show that \eqref{e31} holds for $k = \ell+1$, i.e., 
\begin{equation}\label{e40}P^{(\ell+1)}((a_n+1)_{n\geqslant 0})(m)-P^{(\ell+1)}((a_n)_{n\geqslant 0})(m)\ =\ \binom{m}{\ell+1}, \mbox{ for all }m\geqslant 0.\end{equation}
To do so, we induct on $m$. For $m = 0$, both sides of \eqref{e40} are equal to $0$. Inductive hypothesis: suppose that \eqref{e40} is true for all $m\leqslant s$ for some $s\geqslant 0$. We have
\begin{align*}
&P^{(\ell+1)}((a_n+1)_{n\geqslant 0})(s+1)-P^{(\ell+1)}((a_n)_{n\geqslant 0})(s+1)\\
\ =\ &(P^{(\ell+1)}((a_n+1)_{n\geqslant 0})(s)+P^{(\ell)}((a_n+1)_{n\geqslant 0})(s))\\
&\ \qquad -(P^{(\ell+1)}((a_n)_{n\geqslant 0})(s)+P^{(\ell)}((a_n)_{n\geqslant 0})(s))\\
\ =\ &(P^{(\ell+1)}((a_n+1)_{n\geqslant 0})(s)-P^{(\ell+1)}((a_n)_{n\geqslant 0})(s))\\
&\ \qquad + (P^{(\ell)}((a_n+1)_{n\geqslant 0})(s)-P^{(\ell)}((a_n)_{n\geqslant 0})(s))\\
\ =\ &\binom{s}{\ell+1} + \binom{s}{\ell}\ =\ \binom{s+1}{\ell+1}.
\end{align*}
This completes our proof. 
\end{proof}

The proof of the following lemma is similar to those of Lemmas \ref{kl1} and \ref{kl2}, so we move the proof to the appendix. 

\begin{lem}\label{kl3}
For all $k, \ell\geqslant 0$, we have
\begin{equation}\label{e42}P^{(k)}((F_{n+2})_{n\geqslant 0})(\ell)\ =\ P^{(k)}((F_n)_{n\geqslant 0})(\ell) + P^{(k)}((F_n)_{n\geqslant 0})(\ell+1).\end{equation}
\end{lem}

Before proving Theorem \ref{m4}, we recall \cite[Theorem 1.3]{C2} that gives
\begin{equation}\label{e45}P^{(k)}((F_n)_{n\geqslant 0})(\ell)\ =\ \sum_{j=0}^{\ell-1}\binom{j}{\ell-1-j+k},\mbox{ for }k, \ell\geqslant 0.\end{equation}

\begin{proof}[Proof of Theorem \ref{m4}] We consider the three different cases in the theorem. 

When $k = 1$ and $\ell\geqslant 0$, 
$$\mathcal{A}_{1, 1+\ell} \ =\ \{\emptyset\} \cup \{E\in \mathcal{N}\,:\, E \neq \emptyset, \max E \leqslant 1+\ell, \mbox{ and }\min E > \omega_1(E)\}.
$$
Hence, for a nonempty set $E\in \mathcal{A}_{1, 1+\ell}$, $1\in E$ implies that 
$$1 \ =\ \min E \ >\ \omega_1(E) \ =\ |E|-1,$$
which gives $|E| < 2$. Therefore, $1\in E$ implies that $E = \{1\}$. We can write $\mathcal{A}_{1, 1+\ell}$ as
\begin{align*}
&\{\emptyset\} \cup \{\{1\}\}\cup \{E\subset\{1, 2, \ldots, 1+\ell\}\,:\, E\neq \emptyset \mbox{ and }\min E > |E|\}\\
\ =\ &\{\{1\}\}\cup \{E\subset\{1, 2, \ldots, 1+\ell\}\,:\, E = \emptyset \mbox{ or }\min E > |E|\}.
\end{align*}
We use \cite[Theorem 1]{C1} to obtain $|\mathcal{A}_{1, 1+\ell}| = F_{\ell+2} + 1$. 

When $k\geqslant 2$ and $-k < \ell < 0$, that $a_{k, k+\ell} = F_{k+\ell+1}$ follows from Proposition \ref{kp1}. 

When $k\geqslant 2$ and $\ell \geqslant 0$, we apply Lemma \ref{kl1} with 
$$\begin{cases}\mathbf s_2 &\ =\ (a_{\ell+2, \ell+2})_{\ell\geqslant 0},\\ (t^{\mathbf s_2}_{k-1, \ell})_{\ell\geqslant 0} &\ =\ (a_{k, k+\ell})_{\ell\geqslant 0},\\ (t^{\mathbf s_1}_{k-1, \ell})_{\ell\geqslant 0} &\ =\ (P^{(k-1)}((a_{1, j+1})_{j\geqslant 0})(\ell))_{\ell\geqslant 0},\end{cases}$$
to obtain
\begin{equation}\label{e44}a_{k, k+\ell} - P^{(k-1)}((a_{1, j+1})_{j\geqslant 0})(\ell)\ =\ \sum_{i=0}^{k-2}\binom{\ell}{i}a_{k-i, k-i}\end{equation}
By Theorem \ref{m2} and Proposition \ref{kp1}, we can write \eqref{e44} as 
$$a_{k, k+\ell} - P^{(k-1)}((F_{j+2}+1)_{j\geqslant 0})(\ell)\ =\ 2\sum_{i=0}^{k-2}\binom{\ell}{i}F_{k-i}$$
Using Lemma \ref{kl2}, we arrive at
$$a_{k, k+\ell} \ =\ 2\sum_{i=0}^{k-2}\binom{\ell}{i}F_{k-i} + \binom{\ell}{k-1} + P^{(k-1)}((F_{j+2})_{j\geqslant 0})(\ell),$$
which, due to Lemma \ref{kl3} and then \eqref{e45}, gives
\begin{align*}a_{k, k+\ell} &\ =\ 2\sum_{i=0}^{k-2}\binom{\ell}{i}F_{k-i} + \binom{\ell}{k-1} + P^{(k-1)}((F_{j})_{j\geqslant 0})(\ell)+P^{(k-1)}((F_{j})_{j\geqslant 0})(\ell+1)\\
&\ =\ 2\sum_{i=0}^{k-2}\binom{\ell}{i}F_{k-i} + \binom{\ell}{k-1} + \sum_{j=0}^{\ell-1}\binom{j}{\ell-2-j+k} + \sum_{j=0}^{\ell}\binom{j}{\ell-1-j+k}\\
&\ =\ 2\sum_{i=0}^{k-2}\binom{\ell}{i}F_{k-i} + 2\binom{\ell}{k-1} + \sum_{j=0}^{\ell-1}\binom{j+1}{\ell-1-j+k}\\
&\ =\ 2\sum_{i=0}^{k-2}\binom{\ell}{i}F_{k-i} + 2\binom{\ell}{k-1} + \sum_{j=1}^{\ell}\binom{j}{\ell-j+k},
\end{align*}
as desired. 
\end{proof}

\begin{proof}[Proof of Theorem \ref{c1}]
The proof follows immediately from Theorem \ref{m4} and the identity
\begin{equation}\label{e60}\sum_{i=0}^\ell \binom{\ell}{i}F_{k-i} \ =\ F_{k+\ell}, \mbox{ for all }\ell\geqslant 0\mbox{ and }k\geqslant \ell+2,\end{equation}
which we prove in the appendix. 
\end{proof}

\section{The sequence $(|K_{n}|)_{n\geqslant 2}$}\label{KnFn}

This section proves Theorem \ref{m3} in two different ways. The first proof appears to be more technical but provides more insight (through bijective maps) behind why the identity $|K_n| = F_{n+1}$ holds, while the second proof is relatively shorter as it employs handy formulas. 

\begin{proof}[Proof of Theorem \ref{m3}]
It is easy to verify that $\mathcal{K}_2 = \{\{2\}\}$ and $\mathcal{K}_3 = \{\{3\}\}$. We shall show that $|\mathcal{K}_{n+1}| = |\mathcal{K}_n| + |\mathcal{K}_{n-1}|$ for $n\geqslant 3$. Consider the injective map 
\begin{align*}
    f: \mathcal{K}_n&\ \longrightarrow\ \mathcal{K}_{n+1}\\
    F &\ \longmapsto\ F+1.
\end{align*}
The function $f$ is well-defined because $n+1\in F+1$, $|F+1|\neq 2$, and 
$$\min (F+1)\ =\ \min F + 1 \ > \ \omega_{2,3}(F) + 1\ \geqslant\ \omega_{2,3}(F+1).$$
It remains to show that there is a bijective map $g: \mathcal{K}_{n-1}\longrightarrow \mathcal{K}_{n+1}\backslash f(\mathcal{K}_n)$.  To do so, we prove the following easy claims.
\begin{claim}\label{cl1}
 If $F\in \mathcal{K}_{n-1}$ with $\min F  = 2$ and $|F| > 1$, then $n\geqslant 5$ and $F = \{2, 3, n-1\}$.
\end{claim}

\begin{proof}
Since  $|F|>1$ and $|F|\neq 2$, we know that $|F|\geqslant 3$. We use $2 = \min F > \omega_{2,3}(F)$ to obtain $\omega_{2,3}(F) \leqslant 1$. We deduce from $|F|\geqslant 3$ and $\omega_{2,3}(F)\leqslant 1$ that $F = \{2, 3, n-1\}$ and $n\geqslant 5$. 
\end{proof}

\begin{claim}\label{cl2}
If $F\in \mathcal{K}_{n-1}$ with $\min F = 3$ and $|F| > 1$, then $n\geqslant 6$ and $F = \{3, m, n-1\}$ for some $3 < m < n-1$.
\end{claim}

\begin{proof}
Since $|F| > 1$ and $|F|\neq 2$, we know that $|F|\geqslant 3$. That $\min F = 3$ implies that $3 > \omega_{2,3}(F) = |F| - 1$. Hence, $|F|\leqslant 3$ and so, $|F| = 3$. This gives us the desired conclusion. 
\end{proof}

Claims \ref{cl1} and \ref{cl2} allow us to define $g: \mathcal{K}_{n-1}\longrightarrow \mathcal{K}_{n+1}\backslash f(\mathcal{K}_n)$ as follows:

$$g(F)\ =\ \begin{cases}  \{2,3, n+1\}, &\mbox{ if } F = \{n-1\},\\
\{3, 5, n+1\}, &\mbox{ if } |F| > 1\mbox{ and }\min F = 2,\\
(F\backslash \{3\}+2)\cup \{3\}, &\mbox{ if }|F| > 1\mbox{ and }\min F = 3,\\
 (F+2) \cup \{|F| + 2\}, &\mbox{ if } |F| > 1\mbox{ and }\min F \geqslant 4. \end{cases}$$

\begin{claim}\label{cl3}
We have $g(\mathcal{K}_{n-1})\subset \mathcal{K}_{n+1}$. 
\end{claim}

\begin{proof}
Let $F\in \mathcal{K}_{n-1}$. If $F = \{n-1\}$, then $\max g(F) = \max \{2, 3, n+1\} = n+1$ because $n + 1\geqslant 4$. If $|F| > 1\mbox{ and }\min F = 2$, then Claim \ref{cl1} gives $\max g(F) =  n+1$. If $|F| > 1\mbox{ and }\min F = 3$, then by Claim \ref{cl2}, $\max g(F) = n+1$. 
Lastly, for $|F| > 1\mbox{ and }\min F \geqslant 4$,
$$\max g(F) \ =\ \max ((F+2) \cup \{|F| + 2\})\ =\ n+1,$$
because $|F|\leqslant n-1$. We have shown that $\max g(F) = n+1$. 

It is easy to verify that $|g(F)|\neq 2$. 

Let us show that $\min g(F) > \omega_{2,3}(g(F))$, which is obvious if $F = \{n-1\}$ or if $|F| > 1$ and $\min F = 2$.
If $|F| > 1$ and $\min F = 3$, Claim \ref{cl2} gives
$$\omega_{2,3}(g(F))\ =\ |F|-1\ = \ 2 \ <\ 3 \ =\ \min g(F).$$
If $|F| > 1$ and $\min F \geqslant 4$, then $F\in \mathcal{K}_{n-1}$ implies that $\min F > \omega_{2,3}(F) = |F|$. Then
$$\omega_{2,3}(g(F)) \ =\ |F|+1\ <\ |F|+2\ =\ \min g(F).$$
This completes our proof that $g(F)\in \mathcal{K}_{n+1}$.
\end{proof}

\begin{claim}\label{cl4}
We have $g(\mathcal{K}_{n-1})\cap f(\mathcal{K}_n) = \emptyset$. 
\end{claim}

\begin{proof}
Let $F\in \mathcal{K}_{n-1}$. If $F = \{n-1\}$, it follows from the definition of $g$ that $2\in g(F)$.  However, $2\notin G$, for all $G\in f(\mathcal{K}_n)$. Hence, $g(F)\notin f(\mathcal{K}_n)$.

If $|F| > 1$ and $\min F \in \{2,3\}$, then $\min g(F) = 3$, and the second smallest element of $g(F)$ is at least $5$. On the other hand, for $G\in f(\mathcal{K}_n)$ with $|G|\geqslant 3$ and $\min G = 3$, we know that 
$G-1\in \mathcal{K}_n$, $\min (G-1) = 2$, and $|G-1|\geqslant 3$. By Claim \ref{cl1}, $G-1 = \{2, 3, n\}$ for some $n\geqslant 4$. Hence, $G = \{3, 4, n+1\}$. Since the second smallest element of $G$ is $4$, it must be that $g(F)\neq G$. Therefore, $g(F)\notin f(\mathcal{K}_n)$. 

If $|F| > 1\mbox{ and }\min F \geqslant 4$, suppose, for a contradiction, that $g(F)\in f(\mathcal{K}_n)$, i.e., $(F+2) \cup \{|F| + 2\}\in f(\mathcal{K}_n)$. It follows that $H:= (F+1)\cup \{|F|+1\}\in \mathcal{K}_n$. It follows from $|F|\geqslant 3$ and $\min F\geqslant 4$ that $\min H\geqslant 4$. Hence, 
$$\omega_{2,3}(H) \ =\ |F|+1 \ =\ \min H,$$
which contradicts that $H\in \mathcal{K}_n$. 
\end{proof}

Claims \ref{cl3} and \ref{cl4} guarantee that $g$ is well-defined. Next, we show that $g$ is one-to-one. 
Let 
\begin{align*}
    \mathcal{A} &\ =\ \{F\in \mathcal{K}_{n-1}: |F| > 1\mbox{ and }\min F = 2\},\\
    \mathcal{B} &\ =\ \{F\in \mathcal{K}_{n-1}: |F| > 1\mbox{ and }\min F = 3\},\mbox{ and }\\
    \mathcal{C} &\ =\ \{F\in \mathcal{K}_{n-1}: |F| > 1\mbox{ and }\min F \geqslant  4\}.
\end{align*}
Pick $F_1\in \mathcal{A}, F_2\in \mathcal{B}$, and $F_3\in \mathcal{C}$. Note that $\min g(F_1) = \min g(F_2) = 3$ and $\min g(F_3)\geqslant 5$, so $g(F_1)\neq g(F_3)$ and $g(F_2)\neq g(F_3)$. Furthermore, the second smallest element of $g(F_1)$ is $5$, while the second smallest element of $g(F_2)$ is at least $6$. Hence, $g(F_1)\neq g(F_2)$. We have shown that $g(\mathcal{A}), g(\mathcal{B})$, and $g(\mathcal{C})$ are pairwise disjoint. Finally, $\min g(\{n-1\}) = 2$ guarantees that $g(\{n-1\})\notin g(\mathcal{A})\cup g(\mathcal{B})\cup g(\mathcal{C})$. Therefore, if $F, G\in \mathcal{K}_{n-1}$ such that $g(F) = g(G)$, then $F, G$ must be both equal to $\{n-1\}$ or belong to the same collection $\mathcal{A}$, $\mathcal{B}$, or $\mathcal{C}$. It follows from the definition of $g$ that $F = G$. 

We prove that $g$ is surjective. Pick $E\in \mathcal {K}_{n+1}\backslash f(\mathcal{K}_n)$. Since $E\notin 
f(\mathcal{K}_n)$, $E-1\notin \mathcal{K}_n$, which implies that 
\begin{equation}\label{e50}\min (E-1)\ \leqslant\ \omega_{2,3}(E-1).\end{equation}

Case 1: if $\min E = 2$, then Claim \ref{cl1} gives $n\geqslant 3$ and $E = \{2, 3, n+1\}$. In this case, $g(\{n-1\}) = E$.

Case 2: if $\min E = 3$, then \eqref{e50} gives $|E|\geqslant \omega_{2,3}(E-1) \geqslant 2$, so Claim \ref{cl2} states that $n\geqslant 4$, and $E$ is of the form $\{3, m, n+1\}$ for $4\leqslant m\leqslant n$. Note that $m\neq 4$; otherwise, $E = f(\{2,3, n\})\in f(\mathcal{K}_n)$. Hence, $E = \{3, m, n+1\}$ for $5\leqslant m\leqslant n$.

If $m = 5$, $f(\{2, 3, n-1\}) = E$.

If $m > 5$, $f(\{3, m-2, n-1\}) = E$.

Case 3: if $\min E\geqslant 4$, then 
\begin{equation}\label{e52}
\min E \ >\ |E|.     
\end{equation}Suppose, for a contradiction, that $\min E = 4$. The set $E-1$ is in $\mathcal{K}_n$ because $\max (E-1) = n, \omega_{2,3}(E-1) = |E|-1 < \min (E-1)$, and $|E-1| = |E|\neq 2$. Hence, $E = f(E-1)\in f(\mathcal{K}_n)$, a contradiction. Therefore, $\min E \geqslant 5$. Define $F = (E\backslash \min E)-2$. By \eqref{e50},
\begin{equation}\label{e51}
    4\ \leqslant\ \min (E-1)\ \leqslant\ \omega_{2,3}(E-1) \ =\ |E|.
\end{equation}
It follows from \eqref{e52} and \eqref{e51} that $\min E = |E|+1$.
Furthermore, $F\in \mathcal{K}_{n-1}$ because $\max F = \max E -2 =n-1$, $|F| = |E|-1 \geqslant 3$, and 
$$\omega_{2,3}(F) \ =\ |F| \ =\ |E|-1\ =\ \min E - 2 \ <\ \min F.$$
By the definition of $g$, we have 
$$g(F) \ =\ (E\backslash \min E)\cup \{|F|+2\}\ =\ (E\backslash \min E)\cup \{|E|+1\}\ =\ E.$$
This completes our proof. 
\end{proof}

\begin{proof}[Alternative proof of Theorem \ref{m3}]
The theorem holds for $n\in \{1,2\}$. Assume that $n\geqslant 3$. We consider four different cases.

Case 1: $2, 3\in E$. Since $\max E = n+1\in E$ and $n+1\geqslant 4$, the set $\{2,3,n+1\}$ has size $3$ and is a subset of $E$. It follows from $$2 \ =\ \min E \ >\ \omega_{2,3}(E) \ =\  |E|-2$$
that $|E|\leqslant 3$; hence, $E = \{2, 3, n+1\}$.

Case 2: $2\in E$ and $3\notin E$. In this case, 
$$2 \ =\ \min E\ >\ \omega_{2,3}(E)\ =\ |E|-1.$$
Hence, $|E|\leqslant 2$, so $|E| = 1$ because $|E|\neq 2$. This contradicts that $2, n+1\in E$. 

Case 3: $2\notin E$ and $3\in E$. Then 
$$3 \ =\ \min E \ >\ \omega_{2,3}(E) \ =\ |E|-1,$$
which implies that $|E|\leqslant 3$. Hence, $3, n+1\in E$ and $|E| \neq 2$ imply that $E = \{3, \ell, n+1\}$ for some $4\leqslant \ell\leqslant n$. The number of sets $E$ in this case is $n-3$.

Case 4: $2\notin E$ and $3\notin E$. We need to count the number of sets $E$ satisfying 
\begin{itemize}
\item $E\subset \{4, 5, \ldots, n+1\}$,
\item $\max E = n+1$,
\item $\min E > |E|$, and
\item $|E|\neq 2$.
\end{itemize}
Let $k = \min E$. If $k = n+1$, then $E = \{n+1\}$. Assume that $k < n+1$. The set $E\backslash \{k, n+1\}$ is a subset of $\{k+1, \ldots, n\}$ and has size at most $k-3$ because $\min E > |E|$. Therefore, the number of such sets $E$ is
\begin{align*}
\sum_{k=4}^n\sum_{j=1}^{k-3}\binom{n-k}{j}&\ =\ \sum_{j=1}^{n-3}\sum_{k=j+3}^n \binom{n-k}{j}\ =\ \sum_{j=1}^{n-3}\sum_{r=0}^{n-j-3}\binom{r}{j}\\
&\ =\ \sum_{j=1}^{\lfloor (n-3)/2\rfloor}\sum_{r=0}^{n-j-3}\binom{r}{j}\\
&\ =\ \sum_{j=1}^{\lfloor (n-3)/2\rfloor}\binom{n-j-2}{j+1}\mbox{ by \eqref{a1}}\\
&\ =\ \sum_{j=2}^{\lfloor(n-1)/2\rfloor}\binom{n-1-j}{j}\\
&\ =\ F_n - 1- (n-2) \mbox{ by \eqref{a2}}.
\end{align*}
We have shown that there are $F_n-(n-2)$ sets $E$ (including $\{n+1\}$) that belong to this case. 

Summing over all cases, we obtain the total count of $F_n$, as desired. 
\end{proof}

\section{Appendix}
\begin{proof}[Proof of \eqref{e1}]
For $k\leqslant 1$, let 
\begin{align*}\mathcal{A} &\ := \ \{E\in \mathcal{N}:E=\emptyset \mbox{ or } \min E > \omega_k (E)\}\mbox{ and }\\
\mathcal{B}&\ := \ \mathcal{S}^{NMAX} \cup \{E\in \mathcal{S}^{MAX}: k \in E\}.
\end{align*}
We shall show that $\mathcal{A} = \mathcal{B}$. Let us rewrite $\mathcal{A} = \mathcal{A}_1 \cup \mathcal{A}_2$, where 
\begin{align*}
\mathcal{A}_1 \ &= \ \{E\in \mathcal{N} :k\in E \mbox{ and } \min E > |E|-1\}\mbox{ and }\\
\mathcal{A}_2 \ &= \ \{E\in \mathcal{N}: (k\notin E \mbox{ and } \min E > |E|) \mbox{ or } E = \emptyset\}. 
\end{align*}

First, we show $\mathcal{A}\subset\mathcal{B}$. Pick $F$ in $\mathcal{A}_1$. If $F \in \mathcal{S}^{MAX}$, then $F \in \{E \in \mathcal{S}^{MAX}: k\in E\}\subset \mathcal{B}$ .
If $F \in S^{NMAX}$, then $F \in \mathcal{B}$ by definition. Hence, $\mathcal{A}_1\subset \mathcal{B}$. 
Furthermore, $\mathcal{A}_2$ is a subset of $\mathcal{S}^{NMAX}$. Therefore, $\mathcal{A} = \mathcal{A}_1 \cup \mathcal{A}_2 \subset \mathcal{B}$.

Next, we prove that $\mathcal{B} \subset \mathcal{A}$ by writing
\begin{align*}S^{NMAX} &\ = \ \{E\in\mathcal{N}: \min E > |E| \mbox{ or } E=\emptyset\}\\
&\ =\ \{E\in\mathcal{N}: (k\notin E \mbox{ and }\min E > |E|) \mbox{ or } E=\emptyset\}\\ 
&\qquad \cup \{ E\in\mathcal{N}:k\in E\mbox{ and }\min E>|E|\},\end{align*}
which is clearly in $\mathcal{A}_1\cup \mathcal{A}_2$. 
Moreover, 
$$\{E\in \mathcal{S}^{MAX}\,:\, k\in E\} \ = \ \{E\in \mathcal{N}\,:\, k\in E \mbox{ and } \min E=|E|\} \ \subset\ \mathcal{A}_1.$$ 
Therefore, $\mathcal{B} \subset \mathcal{A},$ which completes our proof.
\end{proof}

\begin{proof}[Proof of Lemma \ref{kl3}]
First, we induct on $k$: for $\ell\geqslant 0$, we have
$$\begin{cases}
    P^{(0)}((F_{n+2})_{n\geqslant 0})(\ell) &\ =\ F_{\ell+2}\\
    P^{(0)}((F_n)_{n\geqslant 0})(\ell) + P^{(0)}((F_n)_{n\geqslant 0})(\ell+1) &\ =\ F_{\ell} + F_{\ell+1}\ =\ F_{\ell+2}.
\end{cases}$$
Hence, \eqref{e42} holds for $k = 0$. Inductive hypothesis: suppose that \eqref{e42} is true for $k\geqslant 0$. We show that it is true for $k + 1$, i.e.,
\begin{equation}\label{e43}P^{(k+1)}((F_{n+2})_{n\geqslant 0})(\ell)\ =\ P^{(k+1)}((F_n)_{n\geqslant 0})(\ell) + P^{(k+1)}((F_n)_{n\geqslant 0})(\ell+1), \mbox{ for all }\ell\geqslant 0.\end{equation}
We induct on $\ell$. Clearly, \eqref{e43} is true for $\ell = 0$, because then both sides are $0$. Inductive hypothesis: suppose that   \eqref{e43} holds for some $\ell\geqslant 0$. 
By the inductive hypotheses, we write
\begin{align*}
P^{(k+1)}((F_{n+2})_{n\geqslant 0})(\ell+1)&\ =\ P^{(k+1)}((F_{n+2})_{n\geqslant 0})(\ell) + P^{(k)}((F_{n+2})_{n\geqslant 0})(\ell)\\
&\ =\ (P^{(k+1)}((F_{n})_{n\geqslant 0})(\ell)+P^{(k+1)}((F_{n})_{n\geqslant 0})(\ell+1)) + \\
&\qquad (P^{(k)}((F_{n})_{n\geqslant 0})(\ell) + P^{(k)}((F_{n})_{n\geqslant 0})(\ell+1))\\
&\ =\ (P^{(k+1)}((F_{n})_{n\geqslant 0})(\ell)+P^{(k)}((F_{n})_{n\geqslant 0})(\ell)) + \\
&\qquad (P^{(k+1)}((F_{n})_{n\geqslant 0})(\ell+1) + P^{(k)}((F_{n})_{n\geqslant 0})(\ell+1))\\
&\ =\ P^{(k+1)}((F_{n})_{n\geqslant 0})(\ell+1) + P^{(k+1)}((F_{n})_{n\geqslant 0})(\ell+2),
\end{align*}
which finishes our proof of \eqref{e43}. 
\end{proof}

\begin{proof}[Proof of \eqref{e60}]
We prove by induction on $\ell$. Base case: \eqref{e60} clearly holds for $\ell = 0$. Inductive hypothesis: suppose that \eqref{e60} holds for some $\ell\geqslant 0$. We show that it holds for $\ell+1$. We have
\begin{align*}
\sum_{i=0}^{\ell+1}\binom{\ell+1}{i}F_{k-i}&\ =\ F_k + F_{k-\ell-1} + \sum_{i=1}^{\ell} \binom{\ell+1}{i}F_{k-i}\\
&\ =\ F_k + F_{k-\ell-1}  + \sum_{i=1}^{\ell}\binom{\ell}{i}F_{k-i} + \sum_{i=1}^{\ell}\binom{\ell}{i-1}F_{k-i}\\
&\ =\ \sum_{i=0}^{\ell}\binom{\ell}{i}F_{k-i} +  \sum_{i=0}^{\ell}\binom{\ell}{i}F_{k-1-i}\\
&\ =\ \left(\sum_{i=0}^{\ell}\binom{\ell}{i}F_{k+1-i} - \sum_{i=0}^{\ell}\binom{\ell}{i}F_{k-1-i}\right)+  \sum_{i=0}^{\ell}\binom{\ell}{i}F_{k-1-i}\\
&\ =\ \sum_{i=0}^{\ell}\binom{\ell}{i}F_{k+1-i} \ =\ F_{k+1+\ell},
\end{align*}
where the last equality is due to the inductive hypothesis. This completes our proof. 
\end{proof}


\ \\
\end{document}